\pdfoutput=1
\RequirePackage{ifpdf}
\ifpdf 
\documentclass[pdftex]{sigma}
\else
\documentclass{sigma}
\fi

\usepackage{mathabx}
\usepackage{bm}

\usepackage[all]{xy}
\SelectTips{cm}{}

\usepackage{tikz}
\usepackage{tikz-cd}
\tikzset{commutative diagrams/.cd}
\tikzstyle{mydashed}=[dash pattern=on 1.5pt off 1.5pt]

\usepackage[newitem,newenum,neverdecrease]{paralist}

\def\id{\mathrm{id}}

\let\map=\xrightarrow
\newcommand{\bdot}{\bm\cdot}

\newcommand{\Kb}{\mathbb{K}}
\newcommand{\Fc}{\mathcal{F}}
\newcommand{\Uc}{\mathcal{U}}
\newcommand{\Vc}{\mathcal{V}}
\newcommand{\Wc}{\mathcal{W}}
\newcommand{\Cs}{\mathsf{C}}
\newcommand{\gVec}{\mathrm{Vec}}
\newcommand{\tone}{\mathbf{1}}
\newcommand{\tp}{\mathbf{p}}

\numberwithin{equation}{section}

\newtheorem{Theorem}{Theorem}[section]
\newtheorem{Lemma}[Theorem]{Lemma}
\newtheorem{Proposition}[Theorem]{Proposition}
 { \theoremstyle{definition}
\newtheorem{Definition}[Theorem]{Definition}
\newtheorem{Example}[Theorem]{Example} }

\begin{document}

\allowdisplaybreaks

\newcommand{\arXivNumber}{2003.11127}

\renewcommand{\PaperNumber}{066}

\FirstPageHeading

\ShortArticleName{Dendriform Algebras Relative to a Semigroup}

\ArticleName{Dendriform Algebras Relative to a Semigroup}

\Author{Marcelo AGUIAR}

\AuthorNameForHeading{M.~Aguiar}

\Address{Department of Mathematics, Cornell University, Ithaca, NY 14853, USA}
\Email{\href{mailto:maguiar@math.cornell.edu}{maguiar@math.cornell.edu}}
\URLaddress{\url{http://www.math.cornell.edu/~maguiar}}

\ArticleDates{Received April 09, 2020, in final form June 29, 2020; Published online July 11, 2020}

\Abstract{Loday's \emph{dendriform} algebras and its siblings \emph{pre-Lie} and \emph{zinbiel} have received attention over the past two decades. In recent literature, there has been interest in a~gene\-ralization of these types of algebra in which each individual operation is replaced by a family of operations indexed by a~fixed semigroup $S$. The purpose of this note is twofold. First, we add to the existing work by showing that a~similar extension is possible already for the most familiar types of algebra: commutative, associative, and Lie. Second, we show that these concepts arise naturally and in a unified manner from a categorical perspective. For this, one simply has to consider the standard types of algebra but in reference to the monoidal category of $S$-graded vector spaces.}


\Classification{17A30; 18C40; 18M05}

\section{Introduction}\label{s:intro}

This note is concerned with various types of algebra: the familiar commutative, associative, and Lie algebras, as well as the related notions of zinbiel, dendriform and pre-Lie algebras. The former two were introduced by Loday \cite{Loday01}, the latter is more classical and goes back to Gerstenhaber \cite{Gerstenhaber63} and Vinberg \cite{Vinberg63}. To the former list one may also add Poisson algebras (which combine commutative and Lie into one structure), and to the latter, the pre-Poisson algebras of \cite{Aguiar00}.
The definitions of all these types of algebra may be found in \cite{Loday01} and \cite{Aguiar00}, whose notation we follow.

Very recently, there has been a substantial amount of work along the following lines. Let~$S$ be a semigroup. For each of the above types of algebra, one wishes to construct another, in which each of the defining operations $\mu$ is replaced by a family $\{\mu_\alpha\}_{\alpha\in S}$ of operations indexed by~$S$. The point is to replace the defining axioms for each type of algebras by a suitable set of new axioms. When~$S$ is the terminal semigroup (the trivial monoid), one should recover the original type of algebra. In the literature, these new types are called \emph{family algebras}. In this note, we employ the terminology \emph{$S$-relative} algebras for a closely related notion. The relationship is explained in Section~\ref{s:cat}.

The present literature has carried out this program on a case-by-case basis, by treating each type of algebra separately. Thus, a definition of dendriform family algebras was introduced in~\cite{ZG19}, and definitions of pre-Lie, zinbiel and pre-Poisson family algebras were introduced in~\cite{MZ20}. Related notions have been considered in \cite{Foissy20,ZGM20a,ZGM20b}.
Yet, similar analogs of commutative, associative, and Lie algebras appear to be missing from the literature.

The purpose of this note is to provide a simple uniform perspective on the matter.
This perspective fulfills the above program in one broad stroke, by showing how to arrive at definitions for all of the types of algebra above, and in fact for any type of algebra defined from a linear operad, relative to~$S$.
It also provides unified proofs of the basic relationships between the various types of algebra.
The simple trick is to upgrade the familiar definitions and proofs by formulating them in the setting of the monoidal category of $S$-graded vector spaces. More precisely, one has to work with the Kleisli coalgebras of the adjunction between vector spaces and $S$-graded vector spaces. This is done in Section~\ref{s:cat}.

As an added bonus, we present definitions of $S$-relative commutative, associative, and Lie algebras. This is done at first in Section \ref{s:relS} by hand, and then recovered from the general perspective in Section \ref{ss:uniform}.

Section \ref{s:variant} discusses other settings in which the categorical approach applies. In Section~\ref{ss:dimonoidal} we sketch the interesting possibility of extending these considerations to dimonoidal categories. This achieves the additional goal of incorporating yet another variant of the dendriform notion in the literature, the matching dendriform algebras of \cite{ZGG20}.

We wish to cite two additional papers that tie to the origin of the subject: \cite{EGP07} and \cite{Guo09}.
First, recall the connection between associative algebras and dendriform algebras afforded by Rota--Baxter operators from \cite{Aguiar00}. This was extended to the $S$-relative context in \cite{ZG19} employing Guo's notion of \emph{Rota--Baxter family of operators}. The notion appeared in \cite[p.~541]{EGP07} after a suggestion by Guo, who further discussed it in \cite[Example~1.3(d)]{Guo09}. Our categorical perspective also incorporates Rota--Baxter operators and provides proofs of the results that relate them to the various types of algebra.

This work does not exhaust all possible ramifications. In particular, we only touch briefly upon aspects related to free algebras in Example~\ref{eg:dend-family-matching}. We do not mention dendriform trialgebras~\cite{LR04}, their connection to Rota--Baxter operators with a weight \cite{EF02}, or the algebra types in \cite{BBGN13,GK13,GK14}. Nevertheless, we hope the note is of use in the further development of the subject.

We work with vector spaces over a field $\Kb$. All operations under consideration are binary and $\Kb$-bilinear. As above, $S$ denotes a fixed associative semigroup, occasionally with a unit element (a monoid), and occasionally commutative. The operation on~$S$ is denoted by juxtaposition.

\section{Associative, commutative and Lie relative to a semigroup}\label{s:relS}

In this section we provide ad hoc definitions of $S$-relative associative, commutative, Lie and Poisson algebras. The operations are indexed by pairs of elements of a semigroup $S$. We check that they relate to the family algebras in the literature (dendriform, zinbiel, pre-Lie and pre-Poisson) in the expected manner. Section~\ref{s:cat} presents these notions as particular examples of a~general construction.

\subsection{Associative and dendriform algebras}\label{ss:assoc}

Let $S$ be an associative semigroup.

\begin{Definition}\label{d:assoc} An $S$-relative associative algebra consists of a vector space $A$ equipped with an operation $\bdot_{\alpha,\beta}$ for each pair $(\alpha,\beta)\in S^2$ and such that
\begin{gather}\label{eq:assoc}
(x\bdot_{\alpha,\beta} y)\bdot_{\alpha\beta,\gamma} z = x\bdot_{\alpha,\beta\gamma}(y\bdot_{\beta,\gamma} z)
\end{gather}
for all $x,y,z\in A$, $\alpha,\beta,\gamma\in S$.
\end{Definition}

\begin{Example}\label{eg:cocycle}Let $c\colon S\times S\to\Kb$ be a semigroup $2$-cocycle
\[
c(\alpha,\beta) c(\alpha\beta,\gamma) = c(\alpha,\beta\gamma) c(\beta,\gamma).
\]
We may turn any associative algebra $A$ into an $S$-relative associative algebra defining
\[
x\bdot_{\alpha,\beta} y = c(\alpha,\beta) xy.
\]
\end{Example}

Consider now a vector space $D$ equipped with two operations $\prec_\alpha$ and $\succ_\alpha$ for each $\alpha\in S$ and such that
\begin{subequations}
\begin{gather}\label{eq:dend1}
(x\prec_\alpha y)\prec_\beta z = x\prec_{\alpha\beta} (y\prec_\beta z + y\succ_\alpha z),\\
\label{eq:dend2}
(x \succ_\alpha y)\prec_\beta z = x\succ_\alpha (y\prec_\beta z),\\
\label{eq:dend3}
(x\prec_\beta y + x\succ_\alpha y)\succ_{\alpha\beta} z = x\succ_\alpha (y\succ_\beta z),
\end{gather}
\end{subequations}
for all $x,y,z\in D$ and $\alpha,\beta\in S$. $D$ is then called a \emph{dendriform family algebra} in \cite{MZ20,ZG19,ZGM20a,ZGM20b}.

\begin{Proposition}\label{p:dend-assoc}Let $D$ be as above. Defining
\begin{gather}\label{eq:dend-assoc}
x\bdot_{\alpha,\beta} y = x\succ_\alpha y + x\prec_\beta y
\end{gather}
turns $D$ into an $S$-relative associative algebra.
\end{Proposition}
\begin{proof}In \eqref{eq:dend1} replace $(\alpha,\beta)$ for $(\beta,\gamma)$. In \eqref{eq:dend2} replace $\beta$ for $\gamma$ and keep $\alpha$. Adding these equations to \eqref{eq:dend3} one obtains \eqref{eq:assoc}.
\end{proof}

\subsection{Commutative and zinbiel algebras}\label{ss:comm}

Assume now that the semigroup $S$ is commutative.

\begin{Definition}\label{d:comm} An $S$-relative commutative algebra is an $S$-relative associative algebra $A$ that satisfies in addition
\begin{gather}\label{eq:comm}
x\bdot_{\alpha,\beta} y = y\bdot_{\beta,\alpha} x
\end{gather}
for all $x,y\in A$, $\alpha,\beta\in S$.
\end{Definition}

Consider now a vector space $Z$ equipped with an operation $\ast_\alpha$ for each $\alpha\in S$ and such that
\begin{gather}\label{eq:zinb}
x\ast_{\alpha} (y\ast_{\beta} z) = (x\ast_{\alpha} y)\ast_{\alpha\beta} z + (y\ast_{\beta}x)\ast_{\alpha\beta} z
\end{gather}
for all $x,y,z\in Z$ and $\alpha,\beta\in S$. $Z$ is then called a \emph{left zinbiel family algebra} in \cite{MZ20}.

\begin{Lemma}\label{l:zinb}Let $Z$ be as above. Then
\begin{gather}\label{eq:zinb-cor}
x\ast_{\alpha} (y\ast_{\beta} z) = y\ast_{\beta} (x\ast_{\alpha} z)
\end{gather}
for all $x,y,z\in A$, $\alpha,\beta\in S$.
\end{Lemma}
\begin{proof}
Replacing $(x,y)$ for $(y,x)$ and $(\alpha,\beta)$ for $(\beta,\alpha)$ in \eqref{eq:zinb} one finds
\[
y\ast_{\beta} (x\ast_{\alpha} z) = (y\ast_{\beta} x)\ast_{\beta\alpha} z + (x\ast_{\alpha} y)\ast_{\beta\alpha} z.
\]
Comparing to \eqref{eq:zinb} and recalling that $S$ is commutative, \eqref{eq:zinb-cor} follows.
\end{proof}

\begin{Proposition}\label{p:principle-family}
Let $D$ and $Z$ be a dendriform and a zinbiel family algebra, respectively.
\begin{enumerate}[$(i)$]\itemsep=0pt
\item\label{it:dend-zinb} Suppose
$x\succ_{\alpha} y = x\prec_{\alpha} y$ for all $x,y\in D$, $\alpha\in S$.
Then defining
$
x\ast _{\alpha} y = x\succ_{\alpha} y
$
turns $D$ into a left zinbiel family algebra.
\item\label{it:zinb-dend} Defining
 $
x\prec _{\alpha} y = y\ast _{\alpha} x
$ and
$
x\succ _{\alpha} y = x \ast _{\alpha} y
$
turns $Z$ into a dendriform family algebra. Moreover,
$
x\succ _{\alpha} y = y\prec _{\alpha}x.
$
\end{enumerate}
\end{Proposition}
\begin{proof}
Consider \eqref{it:zinb-dend}: \eqref{eq:dend1} and \eqref{eq:dend3} follow from \eqref{eq:zinb} and \eqref{eq:dend2} follows from \eqref{eq:zinb-cor}. The proof of \eqref{it:dend-zinb} is similar
(and appears in \cite[Proposition 5.2]{MZ20}).
\end{proof}

\begin{Proposition}\label{p:zinb-comm}
Let $Z$ be as above. Defining
\begin{gather}\label{eq:zinb-comm}
x\bdot_{\alpha,\beta} y = x\ast_\alpha y + y\ast_\beta x
\end{gather}
turns $Z$ into an $S$-relative commutative algebra.
\end{Proposition}
\begin{proof}
It follows from Propositions \ref{p:dend-assoc} and \ref{p:principle-family}\eqref{it:zinb-dend} that \eqref{eq:zinb-comm} turns $Z$ into an $S$-relative associative algebra. Axiom \eqref{eq:comm} for commutativity follows immediately from \eqref{eq:zinb-comm}.
\end{proof}

\subsection{Lie and pre-Lie algebras}\label{ss:lie}

We continue to assume that $S$ is commutative.

\begin{Definition}\label{d:lie} An $S$-relative Lie algebra consists of a vector space $L$ equipped with an operation $[-,-]_{\alpha,\beta}$ for each pair $(\alpha,\beta)\in S^2$ and such that
\begin{subequations}
\begin{gather}\label{eq:skew}
[x,y]_{\alpha,\beta} + [y,x]_{\beta,\alpha} = 0,\\
\label{eq:jacobi}
\big[[x,y]_{\alpha,\beta},z\big]_{\alpha\beta,\gamma} + \big[[z,x]_{\gamma,\alpha},y\big]_{\gamma\alpha,\beta} + \big[[y,z]_{\beta,\gamma},x\big]_{\beta\gamma,\alpha} = 0,
\end{gather}
\end{subequations}
for all $x,y,z\in L$, $\alpha,\beta,\gamma\in S$.
\end{Definition}

Consider now a vector space $P$ equipped with an operation $\circ_\alpha$ for each $\alpha\in S$ and such that
\begin{gather}
\label{eq:prelie}
 x\circ_\alpha (y\circ_\beta z) - (x\circ_\alpha y)\circ_{\alpha\beta} z = y \circ_{\beta} (x \circ_{\alpha} z) - (y \circ_{\beta} x)\circ_{\beta\alpha} z
\end{gather}
for all $x,y,z\in P$ and $\alpha,\beta\in S$. $P$ is then called a \emph{left pre-Lie family algebra} in \cite{MZ20}.

\begin{Proposition}\label{p:prelie-lie}Let $P$ be as above. Defining
\begin{gather}\label{eq:prelie-lie}
[x,y]_{\alpha,\beta} = x\circ_\alpha y - y\circ_\beta x
\end{gather}
turns $P$ into an $S$-relative Lie algebra.
\end{Proposition}
\begin{proof}To establish \eqref{eq:jacobi} one employs the $3$ instances of \eqref{eq:prelie} obtained from it by cyclic permutations of $(x,y,z)$ and $(\alpha,\beta,\gamma)$. Axiom \eqref{eq:skew} follows immediately from \eqref{eq:prelie-lie}.
\end{proof}

\subsection{Poisson and pre-Poisson algebras}

We continue to assume that $S$ is commutative.

\begin{Definition}\label{d:poisson} An $S$-relative Poisson algebra consists of structures of $S$-relative commutative and Lie algebras on the same vector space $A$ and such that
\begin{gather*}
[x, y\bdot_{\beta,\gamma} z]_{\alpha,\beta\gamma}=[x,y]_{\alpha,\beta}\bdot_{\alpha\beta,\gamma} z+y\bdot_{\beta,\alpha\gamma}[x,z]_{\alpha,\gamma}
\end{gather*}
for all $x,y,z\in A$, $\alpha,\beta,\gamma\in S$.
\end{Definition}

Consider now a vector space $B$ equipped with two operations $\circ_\alpha$ and $\ast_\alpha$ for each $\alpha\in S$ that turn it into a left pre-Lie family algebra and a left zinbiel family algebra, respectively, and are such that
\begin{gather*}
 (x\circ_\alpha y - y \circ_\beta x)\ast_{\alpha\beta} z = x\circ_\alpha (y\ast_\beta z) - y\ast_\beta (x\circ_\alpha z),\\
 (x\ast_\alpha y + y\ast_\beta x)\circ_{\alpha\beta} z = x\ast_\alpha (y\circ_\beta z) + y\ast_\beta (x\circ_\alpha z),
\end{gather*}
for all $x,y,z\in P$ and $\alpha,\beta\in S$. $B$ is then called a \emph{left pre-Poisson family algebra} in \cite{MZ20}.

\begin{Proposition}\label{p:prepoisson-poisson}
Let $B$ be as above. With the operations \eqref{eq:zinb-comm} and \eqref{eq:prelie-lie}, $B$ becomes
an $S$-relative Poisson algebra.
\end{Proposition}

\section{The categorical perspective}\label{s:cat}

This section casts the types of algebra from Section~\ref{s:relS} in a general unified setting. The semigroup~$S$ gives rise to the monoidal category of $S$-graded vector spaces. The various types of algebra, and in fact, any type of algebra defined from a linear operad, may be formulated in this setting. This yields a unified approach to the definitions and basic results in the subject.

\subsection{Types of monoid in a monoidal category}\label{ss:mon}

Let $(\Cs,\bullet)$ be monoidal category, not necessarily unital. One may then consider associative monoids in~$\Cs$. An associative monoid is an object~$A$ with a map $\mu\colon A\bullet A\to A$ in $\Cs$ and such that the diagram below commutes.
\begin{gather}\label{eq:assoc-mon}
\begin{gathered}
\xymatrix{
A \bullet A \bullet A\ar[r]^-{\id\bullet\mu}
\ar[d]_{\mu\bullet\id} & A\bullet A \ar[d]^{\mu} \\
A\bullet A\ar[r]_-{\mu} &A.
}
\end{gathered}
\end{gather}
If the monoidal category possesses a unit object, one may consider unital associative monoids.
If the monoidal category is symmetric, one may consider commutative, Lie and Poisson monoids. For more details on the preceding points, see for instance \cite[Sections~1.2.1, 1.2.6 and~1.2.10]{am10}.

One may consider other types of monoid \cite[Section 4.1.1]{am10}. For example, a dendriform monoid in a linear monoidal category $\Cs$ is an object $D$ with maps
\[
\prec \colon \ D \bullet D \to D, \qquad
\succ \colon \ D \bullet D \to D
\]
in $\Cs$ and such that the diagrams below commute, where ${\bdot} = {\prec} + {\succ}$,
\begin{gather}\label{eq:dend-mon}
\begin{gathered}
\xymatrix{
D \bullet D \bullet D \ar[r]^-{\id \bullet \bdot}\ar[d]_{\prec \bullet \id} & D \bullet D \ar[d]^{\prec}\\
D \bullet D \ar[r]_-{\prec} & D,
}
\end{gathered}
\qquad
\begin{gathered}
\xymatrix{
D \bullet D \bullet D \ar[r]^-{\id \bullet \prec}\ar[d]_{\succ \bullet \id} & D \bullet D \ar[d]^{\succ}\\
D \bullet D \ar[r]_-{\prec} & D,
}
\end{gathered}
\qquad
\begin{gathered}
\xymatrix{
D \bullet D \bullet D \ar[r]^-{\id \bullet \succ}\ar[d]_{\bdot \bullet \id} & D \bullet D \ar[d]^{\succ}\\
D \bullet D \ar[r]_-{\succ} & D.
}
\end{gathered}
\end{gather}

If the linear monoidal category is symmetric, one may consider zinbiel, pre-Lie and pre-Poisson monoids. More generally, for any operad $\tp$ (in the category of vector spaces) one may consider $\tp$-monoids in a linear symmetric monoidal category~$\Cs$ \cite[Section~4.2]{am10}.

In a linear monoidal category, one may also consider Rota--Baxter operators of various types. For example, a Rota--Baxter operator on an associative monoid $(A,\mu)$ is a map $R\colon A\to A$ such that the diagram below commutes:
\begin{gather}\label{eq:RB-mon}
\begin{gathered}
\xymatrix@C+1.5pc{
A \bullet A \ar[r]^-{R\bullet\id+\id\bullet R} \ar[d]_{R\bullet R} & A\bullet A \ar[r]^{\mu} & A \ar[d]^{R} \\
A\bullet A \ar[rr]_-{\mu} & & A.
}
\end{gathered}
\end{gather}
Weighted Rota--Baxter operators may also be considered.

The basic principles relating the various types of algebra remain true at this more general level. We list a handful of them. For the proofs one simply formulates the standard arguments in terms of commutative diagrams.

\begin{Proposition}\label{p:principle-mon} Let $\Cs$ be a linear monoidal category.
\begin{enumerate}[$(i)$]\itemsep=0pt
\item\label{it:dend-assoc} If $(D,\prec,\succ)$ is a dendriform monoid in $\Cs$, then defining
\[
\big(D\bullet D \map{\bdot} D\big) = \big(D\bullet D \map{\succ} D\big) +
\big(D\bullet D \map{\prec} D\big)
\]
turns $D$ into an associative monoid in $\Cs$.
\item\label{it:baxter-dend}
 If $R$ is a Rota--Baxter operator on an associative monoid $(A,\mu)$ in $\Cs$, then defining
\begin{gather*}
\big(A\bullet A \map{\prec} A\big) = \big(A\bullet A \map{\id\bullet R} A\bullet A \map{\mu} A\big),
\qquad \text{and}\\
\big(A\bullet A \map{\succ} A\big) = \big(A\bullet A \map{R\bullet\id} A\bullet A \map{\mu} A\big),
\end{gather*}
turns $A$ into a dendriform monoid in $\Cs$.
\end{enumerate}
\end{Proposition}

\begin{Proposition}\label{p:principle-smon}
Let $\Cs$ be a linear symmetric monoidal category, with symmetry
\[
\sigma\colon \ X\bullet Y\to Y\bullet X.
\]
\begin{enumerate}[$(i)$]\itemsep=0pt
\item If $(D,\prec,\succ)$ is a dendriform monoid in $\Cs$, then defining
\[
\big(D\bullet D \map{\circ} D\big) = \big(D\bullet D \map{\succ} D \big) -
\big(D\bullet D\map{\sigma} D\bullet D \map{\prec} D\big)
\]
turns $D$ into a left pre-Lie monoid in $\Cs$.
\item Let $(D,\prec,\succ)$ be a dendriform monoid in $\Cs$ for which
\[
\big(D\bullet D \map{\succ} D\big) = \big(D\bullet D\map{\sigma} D\bullet D \map{\prec} D\big).
\]
Then defining
\[
\big(D\bullet D \map{\ast} D\big) = \big(D\bullet D \map{\succ} D\big)
\]
turns $D$ into a left zinbiel monoid in $\Cs$.
\item Let $(Z,\ast)$ be a left zinbiel monoid in $\Cs$. Then defining
 \[
\big(Z\bullet Z \map{\prec} Z\big) = \big(Z\bullet Z \map{\sigma} Z\bullet Z \map{\ast} Z\big)
\qquad \text{and}\qquad
\big(Z\bullet Z \map{\succ} Z\big) = \big(Z\bullet Z \map{\ast} Z\big),
\]
turns $Z$ into a dendriform monoid in $\Cs$. Moreover,
\[
\big(Z\bullet Z \map{\succ} Z\big) = \big(Z\bullet Z\map{\sigma} Z\bullet Z \map{\prec} Z\big).
\]
\end{enumerate}
\end{Proposition}

\subsection[$S$-graded spaces]{$\boldsymbol{S}$-graded spaces}\label{ss:Sgraded}

Let $S$ be an associative semigroup.

An $S$-\emph{graded space} is a collection $\Vc$ of vector spaces $V_\alpha$, one for each $\alpha$ in $S$. A morphism of $S$-graded spaces is similarly defined in terms of a collection of linear maps. Let $\gVec_S$ denote the resulting category.

Given $S$-graded spaces $\Vc$ and $\Wc$, their \emph{Cauchy product} $\Vc\bullet \Wc$ is defined by
\begin{gather}\label{eq:cauchy}
(\Vc\bullet \Wc)_\gamma = \bigoplus_{\gamma=\alpha\beta} V_\alpha\otimes W_\beta.
\end{gather}
This turns $\gVec_S$ into a (not necessarily unital) monoidal category.

If $S$ is commutative, the monoidal category $\gVec_S$ is symmetric with symmetry $\Vc\bullet \Wc\to \Wc\bullet \Vc$ defined by assembling the switch maps
\[
V_\alpha\otimes W_\beta \to W_\beta\otimes V_\alpha, \qquad x\otimes y \mapsto y\otimes x,
\]
into a map $(\Vc\bullet \Wc)_\gamma \to (\Wc\bullet \Vc)_\gamma$, where $\gamma=\alpha\beta=\beta\alpha$.

If $S$ is a monoid with unit element $\omega$, then $\gVec_S$ is unital with unit object $\tone$ defined by
\[
\tone_\alpha = \begin{cases}
\Kb & \text{if $\alpha=\omega$}, \\
0 & \text{otherwise}.
\end{cases}
\]

Given a vector space $V$, let $\Uc(V)$ be the $S$-graded space defined by
\[
\Uc(V)_\alpha = V
\]
for all $\alpha\in S$. We say that an $S$-graded space of the form $\Uc(V)$ is \emph{uniform}.

The Cauchy product of two uniform $S$-graded spaces need not be uniform.

\subsection[$S$-relative algebras]{$\boldsymbol{S}$-relative algebras}\label{ss:uniform}

\begin{Definition}\label{d:Srel-alg}
Let $S$ be a commutative semigroup and $\tp$ an operad. An \emph{$S$-relative $\tp$-algebra} is a $\tp$-monoid in $\gVec_S$ for which the underlying $S$-graded space is uniform.
\end{Definition}

When the operad is nonsymmetric (or more precisely, the symmetrization of a nonsymmetric operad), the semigroup $S$ in Definition \ref{d:Srel-alg} is merely required to be associative. The operads whose algebras are associative and dendriform algebras are nonsymmetric. In particular, the notions of associative and dendriform algebras are defined in relation to any associative semigroup $S$. If $S$ is a monoid, the notion of $S$-relative unital associative algebra is defined. For the notions of $S$-relative commutative and Lie algebras to be defined, the semigroup $S$ should be commutative.

The $S$-relative types of algebra defined in Section \ref{s:relS} are now seen to arise in this general manner.

\begin{Proposition}\label{p:assoc-cat}
 Let $S$ be an associative semigroup. The two notions of $S$-relative associative algebra in Definitions~{\rm \ref{d:assoc}} and~{\rm \ref{d:Srel-alg}} agree.
\end{Proposition}
\begin{proof}Let $\Uc(A)$ be a uniform associative monoid in $\gVec_S$. We have to show it is an $S$-relative algebra in the sense of Definition~\ref{d:assoc}. The multiplication $\Uc(A)\bdot \Uc(A)\to \Uc(A)$ consists of various linear maps
\[
A\otimes A = A_\alpha\otimes A_\beta \to A_{\alpha\beta} = A.
\]
These are the operations $\bdot_{\alpha,\beta}$ in Definition~\ref{d:assoc}. Axiom~\eqref{eq:assoc-mon} translates into axiom~\eqref{eq:assoc}.
\end{proof}

In the same manner, one has the following.

\begin{Proposition}\label{p:comm-cat} Let $S$ be a commutative semigroup. The two notions of $S$-relative commutative $($Lie, or Poisson$)$ monoid in Definitions~{\rm \ref{d:comm}} {\rm (\ref{d:lie}}, or {\rm \ref{d:poisson})} and {\rm \ref{d:Srel-alg}} agree.
\end{Proposition}

\looseness=-1 Suppose $S$ is a monoid with unit element~$\omega$. Definition~\ref{d:Srel-alg} yields the notion of $S$-relative unital associative algebra~$A$. The result is straightforward: $A$ should possess an element $1$ such that
\[
x\bdot_{\alpha,\omega} 1 = x = 1\bdot_{\omega,\alpha} x
\]
for all $x\in A$, $\alpha\in S$.

When $S$ is finite (or more generally when each element of $S$ possesses only a finite number of factorizations), one may in the same manner obtain the notions of $S$-relative coalgebra, $S$-relative bialgebra (if~$S$ is commutative), and their counital versions (if $S$ is a monoid), among others.

\subsection{Forgetting the semigroup}\label{ss:forget}

Let $\gVec$ denote the category of vector spaces. The functor $\Uc$ is part of an adjunction
\[
\begin{tikzcd}[column sep=1cm]
\gVec_S \ar[r, shift left=.6ex, "\Fc"] & \ar[l, shift left=.6ex, "\Uc"] \gVec.
\end{tikzcd}
\]
The left adjoint $\Fc$ is defined by
\[
\Fc(\Vc) = \bigoplus_{\alpha\in S} V_\alpha.
\]
One may thus refer to $\Uc(V)$ as the \emph{cofree} $S$-graded space on the space $V$. The uniform $S$-graded spaces are the Kleisli coalgebras of the adjunction. See \cite[Chapter VI]{maclane98}.

The functor $\Fc$ satisfies
\[
\Fc(\Vc\bullet\Wc)\cong \Fc(\Vc)\otimes\Fc(\Wc).
\]
It also preserves the symmetry of each category. More precisely, it is a linear symmetric strong monoidal functor. For this reason, $\Fc$ sends a monoid of any type in $\gVec_S$ to a monoid (algebra) of the same type in $\gVec$ \cite[Corollary 4.37]{am10}.

Note that on a uniform object $\Vc=\Uc(V)$, we have
\[
\Fc(\Vc) = V\otimes \Kb S.
\]
This means that if $V$ is an $S$-relative algebra of a given type, then $V\otimes \Kb S$ is an ordinary algebra of the same type. This fact has been observed in a few special cases in the literature: for dendriform family algebras in \cite[Theorem 2.11]{ZGM20a}, for pre-Lie family algebras in \cite[Theorem~3.11]{MZ20}.

\begin{Example}\label{eg:cocycle2}
Consider the $S$-relative associative algebra in Example \ref{eg:cocycle}. The corresponding associative algebra is the tensor product $A\otimes\Kb S$ between the given associative algebra $A$ and the semigroup algebra of $S$, the latter twisted by the given $2$-cocycle.
\end{Example}

\subsection[$S$-relative dendriform algebras]{$\boldsymbol{S}$-relative dendriform algebras}\label{ss:dendrel}

Let $S$ be an associative semigroup. We now apply Definition \ref{d:Srel-alg} to the dendriform operad. We make use of the explicit description of dendriform monoids in Section \ref{ss:mon}.

\begin{Proposition}\label{p:dendrel}
An $S$-relative dendriform algebra consists of a vector space $D$ equipped with operations $\prec_{\alpha,\beta}$ and $\succ_{\alpha,\beta}$ for each pair $(\alpha,\beta)\in S^2$ and such that
\begin{subequations}
\begin{gather}
\label{eq:dendrel1}
(x\prec_{\alpha,\beta} y)\prec_{\alpha\beta,\gamma} z = x\prec_{\alpha,\beta\gamma} (y\prec_{\beta,\gamma} z + y\succ_{\beta,\gamma} z),\\
\label{eq:dendrel2}
(x \succ_{\alpha,\beta} y)\prec_{\alpha\beta,\gamma} z = x\succ_{\alpha,\beta\gamma} (y\prec_{\beta,\gamma} z),\\
\label{eq:dendrel3}
(x\prec_{\alpha,\beta} y + x\succ_{\alpha,\beta} y)\succ_{\alpha\beta,\gamma} z = x\succ_{\alpha,\beta\gamma} (y\succ_{\beta,\gamma} z),
\end{gather}
\end{subequations}
for all $x,y,z\in A$, $\alpha,\beta,\gamma\in S$.
\end{Proposition}

The general construction of an associative monoid from a dendriform monoid given by Proposition \ref{p:principle-mon}\eqref{it:dend-assoc} yields the following when specialized to uniform objects in $\gVec_S$.

\begin{Proposition}\label{p:dend-assoc-rel}
Let $D$ be an $S$-relative dendriform algebra. Defining
\begin{gather}\label{eq:dend-assoc-gen}
x\bdot_{\alpha,\beta} y = x\succ_{\alpha,\beta} y + x\prec_{\alpha,\beta} y
\end{gather}
turns $D$ into an $S$-relative associative algebra.
\end{Proposition}

The notion of $S$-relative dendriform algebra differs from the notion of `dendriform family' algebra of Section \ref{s:relS}: for the latter, the operations are indexed by a single element of $S$ rather than by a pair. One has the following relation between the two.

\begin{Proposition}\label{p:dend-relative-family}
Let $D$ be an $S$-relative dendriform algebra. Suppose the operations $\prec_{\alpha,\beta}$ are independent of $\alpha$ and the operations $\succ_{\alpha,\beta}$ are independent of $\beta$. Then $D$ is a dendriform family algebra, and vice versa.
\end{Proposition}
\begin{proof}
Under these assumptions, axioms \eqref{eq:dendrel1}--\eqref{eq:dendrel3} specialize to \eqref{eq:dend1}--\eqref{eq:dend3}.
\end{proof}

Note that combining Propositions \ref{p:dend-assoc-rel} and \ref{p:dend-relative-family} we obtain a (better) proof of Proposition \ref{p:dend-assoc}: when the operations depend only on one variable as above,
\eqref{eq:dend-assoc-gen} becomes \eqref{eq:dend-assoc}.

Proposition~\ref{p:dend-relative-family} raises the possibility that different assumptions on the operations might lead to additional variants of the notion of dendriform algebra. For example, suppose all operations~$\prec_{\alpha,\beta}$ and~$\succ_{\alpha,\beta}$ of an $S$-relative dendriform algebra are independent of~$\beta$. Then axioms \eqref{eq:dendrel1}--\eqref{eq:dendrel3} specialize to the following
\begin{gather*}
(x\prec_{\alpha} y)\prec_{\alpha\beta} z = x\prec_{\alpha} (y\prec_{\beta} z + y\succ_{\beta} z),\\
(x \succ_{\alpha} y)\prec_{\alpha\beta} z = x\succ_{\alpha} (y\prec_{\beta} z),\\
(x\prec_{\alpha} y + x\succ_{\alpha} y)\succ_{\alpha\beta} z = x\succ_{\alpha} (y\succ_{\beta} z).
\end{gather*}
This notion, while meaningful, has not been considered in the literature.

If instead the operations of an $S$-relative dendriform algebra satisfy that $\prec_{\alpha,\beta}$ is independent of $\beta$ and $\succ_{\alpha,\beta}$ is independent of $\alpha$, then axioms \eqref{eq:dendrel1}--\eqref{eq:dendrel3} specialize to a still meaningful but rather peculiar set of axioms.

\subsection[$S$-relative Rota--Baxter operators]{$\boldsymbol{S}$-relative Rota--Baxter operators}\label{ss:baxterrel}

We now specialize the notion of Rota--Baxter operator to uniform objects in~$\gVec_S$. According to Proposition~\ref{p:assoc-cat}, a uniform associative monoid in~$\gVec_S$ is the same as an $S$-relative associative algebra. Axiom~\eqref{eq:RB-mon} yields axiom~\eqref{eq:baxterrel} below.

\begin{Definition}\label{d:baxterrel}
Let $A$ be an $S$-relative associative algebra. A Rota--Baxter operator on $A$ is a family of operators $R_\alpha\colon A\to A$ such that
\begin{gather}\label{eq:baxterrel}
R_\alpha(x)\bdot_{\alpha,\beta} R_\beta(y) = R_{\alpha\beta}\bigl(R_\alpha(x)\bdot_{\alpha,\beta} y+x\bdot_{\alpha,\beta} R_\beta(y)\bigr)
\end{gather}
for all $x,y\in A$, $\alpha,\beta\in S$.
\end{Definition}
The case of weighted Rota--Baxter operators is similar.

Consider now the general construction of a dendriform monoid from a Rota--Baxter operator on an associative monoid given by Proposition \ref{p:principle-mon}\eqref{it:baxter-dend}. Specializing to uniform objects, we have the following.

\begin{Proposition}\label{p:baxter-dend-rel}
Let $R$ be a Rota--Baxter operator on an $S$-relative associative algebra $A$. Defining
\[
x\prec_{\alpha,\beta} y= x\bdot_{\alpha,\beta} R_\beta(y)
\qquad \text{and}\qquad
x\succ_{\alpha,\beta} y = R_\alpha(x) \bdot_{\alpha,\beta} y
\]
turns $A$ into an $S$-relative dendriform algebra.
\end{Proposition}

The result in the literature is of a more restricted nature. Suppose the operations $\bdot_{\alpha,\beta}$ are independent of both $\alpha$ and $\beta$ (so~$A$ is an ordinary associative algebra). In this case,
Definition~\ref{d:baxterrel} agrees with the definition given in~\cite[p.~541]{EGP07} and \cite[Example~1.3(d)]{Guo09}. And in the situation of Proposition~\ref{p:baxter-dend-rel}, $\prec_{\alpha,\beta}$ is independent of $\alpha$ and $\succ_{\alpha,\beta}$ is independent of $\beta$. In view of Proposition~\ref{p:dend-relative-family}, we obtain a dendriform family algebra structure on $A$. This is the result in \cite[Theorem~4.4.]{ZG19}.

\subsection[$S$-relative zinbiel, pre-Lie and pre-Poisson algebras]{$\boldsymbol{S}$-relative zinbiel, pre-Lie and pre-Poisson algebras}\label{ss:siblings-rel}

 We record the result of applying Definition \ref{d:Srel-alg} to these operads.

\begin{Proposition}\label{p:siblings-rel}
Let $S$ be a commutative semigroup.
\begin{enumerate}[$(i)$]\itemsep=0pt
\item An $S$-relative left zinbiel algebra consists of a vector space $Z$ equipped with operations $\ast_{\alpha,\beta}$ for each pair $(\alpha,\beta)\in S^2$ and such that
\begin{gather}\label{eq:zinbrel}
x\ast_{\alpha,\beta\gamma} (y\ast_{\beta,\gamma} z) = (x\ast_{\alpha,\beta} y)\ast_{\alpha\beta,\gamma} z + (y\ast_{\beta,\alpha} x)\ast_{\beta\alpha,\gamma} z
\end{gather}
for all $x,y,z\in Z$ and $\alpha,\beta,\gamma\in S$.
\item
An $S$-relative left pre-Lie algebra consists of a vector space $P$ equipped with operations $\circ_{\alpha,\beta}$ for each pair $(\alpha,\beta)\in S^2$ and such that
\begin{gather*}
x\circ_{\alpha,\beta\gamma} (y\circ_{\beta,\gamma} z) - (x\circ_{\alpha,\beta} y)\circ_{\alpha\beta,\gamma} z = y \circ_{\beta,\alpha\gamma} (x \circ_{\alpha,\gamma} z) - (y \circ_{\beta,\alpha} x)\circ_{\beta\alpha,\gamma} z
\end{gather*}
for all $x,y,z\in P$ and $\alpha,\beta,\gamma\in S$.
\item
An $S$-relative left pre-Poisson algebra consists of a vector space $B$ equipped with two operations $\circ_{\alpha,\beta}$ and $\ast_{\alpha,\beta}$ for each pair $(\alpha,\beta)\in S^2$ that turn it into an $S$-relative left pre-Lie and zinbiel algebra, respectively,
and are such that
\begin{gather*}
 (x\circ_{\alpha,\beta} y - y \circ_{\beta,\alpha} x)\ast_{\alpha\beta,\gamma} z = x\circ_{\alpha,\beta\gamma} (y\ast_{\beta,\gamma} z) - y\ast_{\beta,\alpha\gamma} (x\circ_{\alpha,\gamma} z),\\
 (x\ast_{\alpha,\beta} y + y\ast_{\beta,\alpha} x)\circ_{\alpha\beta,\gamma} z = x\ast_{\alpha,\beta\gamma} (y\circ_{\beta,\gamma} z) + y\ast_{\beta,\alpha\gamma} (x\circ_{\alpha,\gamma} z),
\end{gather*}
for all $x,y,z\in B$ and $\alpha,\beta,\gamma\in S$.
\end{enumerate}
\end{Proposition}

The `family algebras' of Section \ref{s:relS} are special cases of the above notions.

\begin{Proposition}\label{p:siblings-relative-family}\
\begin{enumerate}[$(i)$]\itemsep=0pt
\item A zinbiel family algebra is the same as an $S$-relative zinbiel algebra in which the operations~$\ast_{\alpha,\beta}$ are independent of $\beta$.
\item A pre-Lie family algebra is the same as an $S$-relative pre-Lie algebra in which the operations~$\circ_{\alpha,\beta}$ are independent of $\beta$.
\item A pre-Poisson family algebra is the same as an $S$-relative pre-Poisson algebra in which the operations $\ast_{\alpha,\beta}$ and $\circ_{\alpha,\beta}$ are independent of $\beta$.
\end{enumerate}
\end{Proposition}

Specializing Proposition \ref{p:principle-smon} to uniform objects in $\gVec_S$ yields the following.

\begin{Proposition}\label{p:principle-relS}
 Let $D$ be an $S$-relative dendriform algebra and let $Z$ be an $S$-relative left zinbiel algebra.
\begin{enumerate}[$(i)$]\itemsep=0pt
\item Defining
$
x\circ_{\alpha,\beta} y = x\succ_{\alpha,\beta} y - y\prec_{\alpha,\beta} x
$
turns $D$ into an $S$-relative left pre-Lie algebra.
\item Suppose
$
x\succ_{\alpha,\beta} y = y\prec_{\alpha,\beta} y.
$
Then defining
$
x\ast _{\alpha,\beta} y = x\succ_{\alpha,\beta} y
$
turns $D$ into an $S$-relative left zinbiel algebra.
\item Defining
 $
x\prec _{\alpha,\beta} y = y\ast _{\alpha,\beta} x
$ and
$
x\succ _{\alpha,\beta} y = x \ast _{\alpha,\beta} y
$
turns $Z$ into an $S$-relative dendriform algebra. Moreover,
$
x\succ _{\alpha,\beta} y = y\prec _{\alpha,\beta}x.
$
\end{enumerate}
\end{Proposition}

Restricting to the case in which the operations are independent of the second index $\beta$ we obtain corresponding results for `family algebras', including Proposition~\ref{p:principle-family}.

\subsection[Morphisms of $S$-relative algebras]{Morphisms of $\boldsymbol{S}$-relative algebras}\label{ss:mor-rel}

Let $\tp$ be an operad and $S$ an associative semigroup (commutative unless $\tp$ is nonsymmetric).

\begin{Definition}\label{d:mor-rel}A morphism of $S$-relative $\tp$-algebras is a morphism of $\tp$-algebras in $\gVec_S$, that is, a morphism of $S$-graded spaces that preserves the operations.
\end{Definition}

\looseness=-1 A morphism $f\colon \Uc(A)\to\Uc(B)$ then consists of a family of linear maps $f_\alpha\colon A\to B$, one for each $\alpha\in S$, subject to axioms that depend on the type of algebra. For associative algebras, this is
\[
f_{\alpha\beta}(x\bdot_{\alpha,\beta} y) = f_\alpha(x) \bdot_{\alpha,\beta} f_\beta(y).
\]
The conditions are similarly straightforward for the other types of algebra in the previous sections.

One may consider the special case in which the maps $f_\alpha$ are independent of $\alpha$, that is, when the morphism consists of a single map $A\to B$. This is the type of morphism considered in the literature for `family algebras', particularly when `free family algebras' are constructed \cite{MZ20,ZGM20a}.

\section{Additional variants}\label{s:variant}

We discuss other settings in which the categorical approach of Section \ref{s:cat} applies. We also sketch an extension to the case of dimonoidal categories, which allows for an even more general formulation of the notion of dendriform algebra.

\subsection{Employing other monoidal categories}

As discussed in Section \ref{ss:mon}, one may consider the various types of algebra in any linear symmetric monoidal category $\Cs$. One may for example choose for $\Cs$ the category of $H$-comodules over a~bialgebra $H$, and obtain a notion of \emph{$H$-relative dendriform algebra}, as well as of all of the other types. The role of uniform $S$-graded spaces is played in this setting by cofree $H$-comodules. Explicitly, an $H$-relative dendriform algebra is a vector space $D$ equipped with operations
\[
x\succ_{a,b} y \qquad \text{and}\qquad x\prec_{a,b} y.
\]
These expressions belong to $D$ and are linear in each of $x,y\in D$ and $a,b\in H$. The middle axiom in~\eqref{eq:dend-mon}
becomes
\[
\sum (x\succ_{a_2,b_2} y)\prec_{a_1b_1,c} z = \sum x\succ_{a,b_1c_1}(y\prec_{b_2c_2} z)
\]
for $x,y,z\in D$, $a,b,c\in H$, where we have written $\Delta(a) = \sum a_1\otimes a_2$ for the comultiplication of~$H$.

Or one may be interested in the case when $\Cs$ is the category of
$S$-modules over a semi\-group~$S$, or the category of $H$-modules over a bialgebra~$H$, among many others.

A case of potential interest is afforded by the category of $S$-graded spaces with the monoidal structure twisted by an abelian $3$-cocycle, as in \cite[Remark~3.2] {JS93}. For example, the zinbiel axiom~\eqref{eq:zinbrel} becomes
\[
h(\alpha,\beta,\gamma)\, x\ast_{\alpha,\beta\gamma} (y\ast_{\beta,\gamma} z) = (x\ast_{\alpha,\beta} y)\ast_{\alpha\beta,\gamma} z + c(\alpha,\beta)\, (y\ast_{\beta,\alpha} x)\ast_{\beta\alpha,\gamma} z,
\]
where $h\colon S\times S\times S\to \Kb^\times$ and $c\colon S\times S\to\Kb^\times$ satisfy the conditions in \cite[p.~47] {JS93}. The function~$c$ should satisfy in addition $c(\alpha,\beta)=c(\beta,\alpha)$ to ensure that the twisted braiding in~$\gVec_S$ is a~symmetry.

One may again generalize and consider (co)modules over a (co)quasi-bialgebra \cite[Chapter~XV]{Kassel95}.

The gist of this note is that the basic properties of the various types will be true in each case and do not need separate proofs.

\subsection{Employing dimonoidal categories}\label{ss:dimonoidal}

A variant of the notion of dendriform algebra has been introduced by Gao, Guo and Zhang in
\cite[Definition~3.1]{ZGG20}, under the name of \emph{matching dendriform algebras}. The operations of these objects are indexed by elements of a~set~$S$. This notion differs from that of dendriform family algebra, and does not arise as a special case of the notion of $S$-relative dendriform algebra. Nevertheless, there exists a categorical approach to this and even more general notions. We briefly sketch the main ingredients next.

Let $S$ be a \emph{dimonoid}. The set $S$ carries two operations $\dashv$ and $\vdash$ satisfying the axioms given in \cite[Definition 1.1]{Loday01}.
The category $\gVec_S$ carries then two monoidal structures $\oleft$ and $\oright$. Each is defined from one of the operations in $S$ by means of \eqref{eq:cauchy}.
The axioms for $S$ imply that the two structures on $\gVec_S$ are linked by isomorphisms as follows
\begin{subequations}\label{eq:dimon-cat}
\begin{gather}
A\oleft (B\oleft C) \cong (A\oleft B) \oleft C \cong A \oleft (B\oright C),\\
(A\oright B)\oleft C \cong A\oright (B\oleft C),\\
(A\oleft B)\oright C \cong A\oright (B\oright C) \cong (A\oright B)\oright C.
\end{gather}
\end{subequations}
These isomorphisms satisfy certain coherent conditions. (We have not worked this out in detail. They should extend Mac Lane's pentagon.) We say then that $\gVec_S$ is a \emph{dimonoidal category}.\footnote{Dimonoidal categories are not the same as the \emph{linearly distributive categories} of \cite[Section~1]{CS99}, previously called \emph{weakly distributive categories} in~\cite{CS92,CS97}. Some of the coherence conditions in \cite[Section~2.1]{CS92} should be common to both notions. Dimonoidal categories also differ from the \emph{$2$-monoidal categories} of \cite[Chapter~6]{am10}, called \emph{duoidal categories} in more recent literature \cite{Street12}.}

The key point is that in a linear dimonoidal category, one may formulate a notion of dendriform monoid: this is an object $D$ equipped with maps
\[
D\oleft D \map{\prec} D \qquad \text{and}\qquad D\oright D \map{\succ} D
\]
subject to $3$ axioms generalizing~\eqref{eq:dend-mon}. For example, the first of these is the following equality
\begin{gather*}
 \bigl((D\oleft D)\oleft D \cong D\oleft(D\oleft D) \map{\id\oleft\prec}D\oleft D \map{\prec} D\bigr)\\
\qquad\quad{} + \bigl((D\oleft D)\oleft D \cong D\oleft(D\oright D) \map{\id\oleft\succ}D\oleft D \map{\prec} D\bigr)\\
\qquad{} = \bigl((D\oleft D)\oleft D \map{\prec\oleft\id} D\oleft D \map{\prec} D\bigr).
\end{gather*}
The reader can easily write down the other two axioms, working from~\eqref{eq:dend-mon} and~\eqref{eq:dimon-cat}.

One then has the notion of dendriform monoid in the dimonoidal category $\gVec_S$. One may next consider the case in which the underlying object is uniform, as in Section~\ref{ss:uniform}. This yields a notion of dendriform algebra relative to a dimonoid $S$. As in Section~\ref{ss:dendrel}, the operations of such an algebra are indexed by pairs $(\alpha,\beta)\in S^2$. The axioms are as follows
\begin{gather*}
(x\prec_{\alpha,\beta} y)\prec_{\alpha\dashv\beta,\gamma} z = x\prec_{\alpha,\beta\dashv\gamma} (y\prec_{\beta,\gamma} z) + x\prec_{\alpha,\beta\vdash\gamma} (y\succ_{\beta,\gamma} z),\\
(x \succ_{\alpha,\beta} y)\prec_{\alpha\vdash\beta,\gamma} z = x\succ_{\alpha,\beta\dashv\gamma} (y\prec_{\beta,\gamma} z),\\
(x\prec_{\alpha,\beta} y)\succ_{\alpha\dashv\beta,\gamma} z + (x\succ_{\alpha,\beta} y)\succ_{\alpha\vdash\beta,\gamma} z = x\succ_{\alpha,\beta\vdash\gamma} (y\succ_{\beta,\gamma} z).
\end{gather*}
One may proceed and consider the special case in which $\prec_{\alpha,\beta}$ is independent of $\alpha$ and $\succ_{\alpha,\beta}$ is independent ot $\beta$. This yields a notion of dendriform algebra in which the operations are indexed by elements of the dimonoid $S$. The axioms are now
\begin{subequations} \label{eq:dend-family-matching}
\begin{gather}
(x\prec_{\alpha} y)\prec_{\beta} z = x\prec_{\alpha\dashv\beta} (y\prec_{\beta} z) + x\prec_{\alpha\vdash\beta} (y\succ_{\alpha} z),\\
(x \succ_{\alpha} y)\prec_{\beta} z = x\succ_{\alpha} (y\prec_{\beta} z),\\
(x\prec_{\beta} y)\succ_{\alpha\dashv\beta} z + (x\succ_{\alpha} y)\succ_{\alpha\vdash\beta} z = x\succ_{\alpha} (y\succ_{\beta} z).
\end{gather}
\end{subequations}

Finally, we may further specialize in two different ways. First, if the operations of the dimonoid~$S$ are simply
\[
\alpha\dashv\beta = \alpha \qquad \text{and}\qquad \alpha\vdash\beta = \beta,
\]
the above notion is precisely that of matching dendriform algebra. Second, if the operations of the dimonoid~$S$ satisfy
\[
\alpha\dashv\beta=\alpha\vdash\beta
\]
(namely, if $S$ is merely an associative monoid), we simply recover the construction of Section \ref{s:cat}, and the above notion is that of a dendriform family algebra.

In order to develop a corresponding approach to zinbiel and pre-Lie algebras, one should work with dimonoidal categories equipped with isomorphisms
\[
A\oleft B \cong B\oright A,
\]
again subject to coherent conditions. The category $\gVec_S$ constitutes an example when~$S$ is a~\emph{permutative monoid}, as in \cite[Section~1]{Chapoton01}.

\begin{Example}\label{eg:dend-family-matching}
Let $X$ be a set. The free dendriform family algebra on~$X$ (for a given semi\-group~$S$) was described in \cite[Section~3.2]{ZGM20a} and the free matching dendriform algebra on~$X$ (for a given set~$S$) was described in \cite[Section~3.2]{ZGG20}. The underlying set is the same for both algebras. We next observe that the two constructions may be unified by turning the same set into the free algebra defined by axioms \eqref{eq:dend-family-matching} (for a given a dimonoid~$S$). The idea behind all these constructions goes back to Loday \cite[Section~5.5]{Loday01}.

We employ the same setting as in both \cite[Section~3.2]{ZGM20a} and \cite[Section~3.2]{ZGG20}. Let $Y$ be the set of planar rooted binary trees in which the vertices are decorated by elements of $X$ and the internal edges are decorated by elements of $S$. (The leaves are not regarded as vertices. An internal edge joins two vertices.) Let $\widehat{Y}=Y\cup\{e\}$, where $e$ stands for a new symbol that we may think of as the tree with no vertices.

Note that any $t\in Y$ is of the form
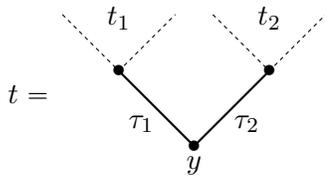
\[
t =
\begin{gathered}
\begin{tikzpicture}[scale=1]
\draw[thick] (-1,1) node {$\bullet$} -- (0,0) node {$\bullet$} -- (1,1) node {$\bullet$};
\draw[mydashed] (-1.75,1.75) -- (-1,1) -- (-0.25,1.75) ;
\draw[mydashed] (0.25,1.75) -- (1,1) -- (1.75,1.75) ;
\draw (0,0) node[below] {$y$};
\draw (-1,1.7) node {$t_1$}; \draw (1,1.7) node {$t_2$};
\draw (-0.7,0.3) node {$\tau_1$}; \draw (0.7,0.3) node {$\tau_2$};
\end{tikzpicture}
\end{gathered}
\]
where $y\in X$ is the label of the root of $t$, the trees $t_1,t_2\in \widehat{Y}$ are the left and right subtrees, and $\tau_1,\tau_2\in S$ are the labels of the internal edges stemming from the root. It is possible that $t_i=e$, in which case there is no vertex in that subtree and the label $\tau_i$ is not defined. We employ similar notation for another tree $s\in Y$: the label of the root is $x\in X$, the left and right subtrees are $s_1$ and $s_2$, the labels of the internal edges at the root are $\sigma_1$ and $\sigma_2$.

The operations on the vector space with basis $Y$ are defined by means of the following recursive formulas
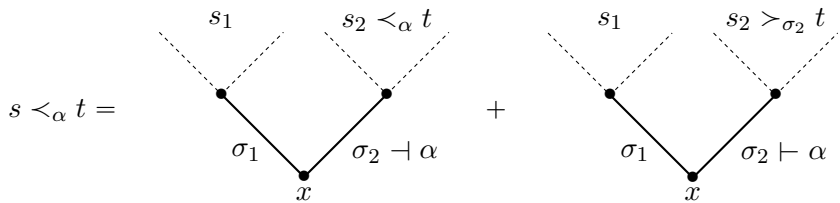
\begin{gather*}
s \prec_{\alpha} t = \quad
\begin{gathered}
\begin{tikzpicture}[scale=1.1]
\draw[thick] (-1,1) node {$\bullet$} -- (0,0) node {$\bullet$} -- (1,1) node {$\bullet$};
\draw[mydashed] (-1.75,1.75) -- (-1,1) -- (-0.25,1.75) ;
\draw[mydashed] (0.25,1.75) -- (1,1) -- (1.75,1.75) ;
\draw (0,0) node[below] {$x$};
\draw (-1,1.9) node {$s_1$}; \draw (1,1.9) node {$s_2\prec_\alpha t$};
\draw (-0.7,0.3) node {$\sigma_1$}; \draw (1.1,0.3) node {$\sigma_2\dashv\alpha$};
\end{tikzpicture}
\end{gathered}
\quad + \quad
\begin{gathered}
\begin{tikzpicture}[scale=1.1]
\draw[thick] (-1,1) node {$\bullet$} -- (0,0) node {$\bullet$} -- (1,1) node {$\bullet$};
\draw[mydashed] (-1.75,1.75) -- (-1,1) -- (-0.25,1.75) ;
\draw[mydashed] (0.25,1.75) -- (1,1) -- (1.75,1.75) ;
\draw (0,0) node[below] {$x$};
\draw (-1,1.9) node {$s_1$}; \draw (1,1.9) node {$s_2\succ_{\sigma_2} t$};
\draw (-0.7,0.3) node {$\sigma_1$}; \draw (1.1,0.3) node {$\sigma_2\vdash\alpha$};
\end{tikzpicture}
\end{gathered}
\\
s \succ_{\alpha} t = \quad
\begin{gathered}
\begin{tikzpicture}[scale=1.1]
\draw[thick] (-1,1) node {$\bullet$} -- (0,0) node {$\bullet$} -- (1,1) node {$\bullet$};
\draw[mydashed] (-1.75,1.75) -- (-1,1) -- (-0.25,1.75) ;
\draw[mydashed] (0.25,1.75) -- (1,1) -- (1.75,1.75) ;
\draw (0,0) node[below] {$y$};
\draw (-1,1.9) node {$s\prec_{\tau_1}t_1$}; \draw (1,1.9) node {$t_2$};
\draw (-1.1,0.3) node {$\alpha\dashv\tau_1$}; \draw (0.7,0.3) node {$\tau_2$};
\end{tikzpicture}
\end{gathered}
\quad + \quad
\begin{gathered}
\begin{tikzpicture}[scale=1.1]
\draw[thick] (-1,1) node {$\bullet$} -- (0,0) node {$\bullet$} -- (1,1) node {$\bullet$};
\draw[mydashed] (-1.75,1.75) -- (-1,1) -- (-0.25,1.75) ;
\draw[mydashed] (0.25,1.75) -- (1,1) -- (1.75,1.75) ;
\draw (0,0) node[below] {$y$};
\draw (-1,1.9) node {$s\succ_\alpha t_1$}; \draw (1,1.9) node {$t_2$};
\draw (-1.1,0.3) node {$\alpha\vdash \tau_1$}; \draw (0.7,0.3) node {$\tau_2$};
\end{tikzpicture}
\end{gathered}
\end{gather*}
The recursions are started by
 \begin{gather*}
 s \prec_{\alpha} e = s, \qquad s\succ_{\alpha} e = 0,\\
 e \prec_{\alpha} t = 0, \qquad e \succ_{\alpha} t = t.
 \end{gather*}
\end{Example}

We close by mentioning that Foissy has introduced a further variant of the notion of dendriform algebra in \cite[Definition 11]{Foissy20}. As for the notion defined by~\eqref{eq:dend-family-matching}, Foissy's notion includes as special cases both dendriform family algebras and matching dendriform algebras.
Does there exist a categorical approach to this notion?

\subsection*{Acknowledgements}

The author thanks the referees for pertinent comments and suggestions. Research partially supported by Simons grant 560656.

\pdfbookmark[1]{References}{ref}
\LastPageEnding

\end{document}